\documentclass{amsart}

\usepackage{amsmath}
\usepackage{amsthm}
\usepackage{amsfonts}
\usepackage{amssymb}

\newcommand\rank{\operatorname{rank}}

\renewcommand\P{{\mathbf{P}}}
\newcommand\E{{\mathbf{E}}}

\newcommand\eps{{\varepsilon}}

\newcommand\Span{\operatorname{Span}}

\newcommand\CE{{\mathcal E}}

\newcommand\CN{{\mathcal N}}

\theoremstyle{plain}
  \newtheorem{theorem}{Theorem}
  
  \newtheorem{problem}[theorem]{Problem}

  \newtheorem{lemma}[theorem]{Lemma}
  \newtheorem{corollary}[theorem]{Corollary}
  
  \newtheorem{question}[theorem]{Question}

\theoremstyle{definition}

  \newtheorem{remark}[theorem]{Remark}

\include{psfig}

\begin{document}

\title{Singular vectors under random perturbation}

\author{Van Vu}
\address{Department of Mathematics, Rutgers, Piscataway, NJ 08854}
\email{vanvu@math.rutgers.edu}

\thanks{V. Vu is supported by research grants DMS-0901216 and AFOSAR-FA-9550-09-1-0167.}

%\thanks{V. Vu is supported by research grants from AFORS and NSF.}

\begin{abstract} 
Computing the first  few  singular vectors of a large matrix is a problem 
that frequently comes up in statistics and numerical analysis. Given the presence of noise, exact calculation is hard to achieve, and the following problem is of importance: 

\vskip2mm

\centerline {\it How much a small perturbation to the matrix  changes the singular vectors ? }

\vskip2mm

Answering this question, classical theorems, such as those of Davis-Kahan and Wedin, give tight estimates for the worst-case scenario. 
In this paper, we show that if the perturbation (noise) is random and our matrix has low rank, then 
better estimates can be obtained. Our method relies on high dimensional geometry and is  different from those used an earlier papers.  \end{abstract}

\maketitle

%\setcounter{tocdepth}{2}
%\tableofcontents
{\it MSC indices: 65F15, 15A42, 62H30}

\section{Introduction} 

An important  problem that appears 
  in various areas of applied mathematics 
  (in particular statistics, computer science and numerical analysis)  is to compute  the first  few  singular vectors of a large matrix. Among others, this problem lies at the heart  of PCA (Principal Component Analysis), which has a 
   very wide range of applications (for many examples, see  \cite{KVbook, LR} 
   and the references therein).

\vskip2mm

The basic setting  of the problem is as follows:

\vskip2mm

\begin{problem}  Given a   matrix $A$ of size $n \times n $  with 
singular values $\sigma_1 \ge  \dots \ge \sigma_n \ge 0$. 
Let $v_1, \dots, v_n$ be the corresponding (unit) singular vectors.  Compute
$v_1, \dots, v_k$, for some $k \le n$.  \end{problem}

Typically  $n$ is  large 
and $k$ is relatively small. As a matter of fact, in many applications  $k$ is   a constant independent of $n$. 
For example, to obtain a visualization of a large set of data, 
one often sets $k=2$ or $3$.  The assumption that $A$ is a square matrix is for convenience and our analysis 
can be carried out with nominal modification for rectangular matrices. 

We use asymptotic notation such as $\Theta, \Omega, O$ under the assumption that 
$n \rightarrow \infty$.  The vectors $v_1, \dots, v_k$ are not unique. However, if  $\sigma_1, \dots, \sigma_k$ are different, then they  are determined up to the sign. 
We assume this is the case  in all discussions. (In fact, as the reader will see,  the gap $\delta_i := \sigma_i -\sigma_{i+1}$ plays a 
crucial role.) For a vector $v$,  $\|v\|$ denotes its $L_2$ norm. For a matrix $A$,  $\| A \|=\sigma_1 (A)$ denotes 
its spectral norm. 

\vskip2mm

\subsection{Classical perturbation bounds} 

The matrix $A$, which represents some sort of  data, is  often perturbed by noise. Thus, one typically 
works with   $A+E$, where $E$ represents the noise. 
 A  natural and important problem is 
to estimate the influence of noise on the vectors $v_1, \dots, v_k$. We denote by $v_1' , \dots, v_k'$ the first $k$ singular vectors of $A+E$. 

 For sake of 
presentation, we restrict ourselves to the case $k=1$ (the first singular vector). 
Our analysis extends  easily in the general case,  discussed in Section 
\ref{section:extension}. 

The following question is of importance 

\begin{question}  \label{question:fundamental} 
When is    $v_1'$  a good approximation of  $v_1$ ? 
\end{question}

A convenient way to measure the distance between two unit vectors  $v$ and $v'$ is to look at 
 $\sin \angle(v, v')$, where $\angle (v, v') $ is the angle between the vectors, taken in $[0, \pi/2]$.  
To make the problem more quantitative, let us fix a small parameter $\eps >0$, 
which represents a desired accuracy. Our question now  is to find a 
sufficient condition  for the matrix $A$ which guarantees that 
$\sin \angle (v_1, v_1') \le \epsilon$.  It has turned out that 
the key parameter to look at is the gap (or separation) 

$$\delta:= \sigma_1 -\sigma_2, $$
between the first and second singular values of $A$.  
Classical  results in numerical linear algebra yield

%We have the following important corollary, which gives a direct answer to Question \ref{question:fundamental}. 

\begin{corollary} \label{epsC1} 
For any given $\eps >0$, there is $C=C(\eps)>0$ such that if $\delta \ge C \|E\|$, then 

$$ \sin \angle(v_1, v'_1 ) \le \epsilon. $$
\end{corollary} 

This follows from a well known result of Wedin

\begin{theorem}  \label{theorem:SS} (Wedin $\sin$ theorem) There is a  positive constant $ C$ such that

\begin{equation}  \label{upper1} \sin \angle(v_1, v'_1 )   \le  C \frac{ \| E\| }{\delta }. \end{equation} 
\end{theorem} 

In the case when $A$ and $A+E$ are hermitian, this statement is a special case of 
the famous  Davis-Kahan $\sin \theta$ theorem. 
Wedin \cite{Wed} extended Davis-Kahan theorem to non-hermitian matrices, resulting 
in a general theorem that contains 
Theorem \ref{theorem:SS} as  a special case 
(see  \cite[Chapter 8]{GLbook}  for more discussion and history).

%\begin{remark}
%The formulation of Theorem \ref{theorem:SS} is slightly redundant, as one can omit the assumption $
%\eqref{lower1}, noticing that the bound \eqref{upper1} is non-trivial only if 
%$\frac{\delta}{\| E\|}$ is sufficiently large. We, however, stick with  this (rather traditional) formulation, as it %emphasizes the condition on $\delta$, which will be the focus of this paper. 
%\end{remark} 

%It is important to notice  that in this theorem the actual values of $\sigma_1$ and $\sigma_2$ are irrelevant. 

Let us consider  the following simple, but illustrative  example \cite{Bbook}. 
Let $A$ be the matrix 
\[ 
\begin{pmatrix}
1+\epsilon & 0 \\
0 & 1-\epsilon 
\end{pmatrix} .\]

Apparently, the singular values of $A$ are $1+\epsilon$ and $1-\epsilon$, with corresponding singular vectors  $(1,0)$ and $(0,1)$.
Let $E$ be 

\[ 
\begin{pmatrix}
-\epsilon & \epsilon \\
\epsilon & \epsilon 
\end{pmatrix} ,\] where $\epsilon$ is a small positive number. 
The perturbed matrix $A+E$ has the form 

\[ 
\begin{pmatrix}
1 & \epsilon \\
\epsilon & 1
\end{pmatrix}. \]

Obviously,  the singular values $A+E$ are also $1+\epsilon$ and $1-\epsilon$. However, the  corresponding 
singular vectors  now are $(\frac{1}{\sqrt 2} ,\frac{1}{\sqrt 2})$ and $(\frac{1}{\sqrt 2}, -\frac{1}{\sqrt 2})$, no matter how small $\epsilon$ is. 
This example shows that the consideration of the gap $\delta$ is necessary, and also that Theorem \ref{theorem:SS} is sharp, up to a 
constant factor. 

\vskip2mm 

\subsection{Random perturbation} 
Noise (or perturbation) represents errors that come from various sources 
which are frequently of   entirely different nature, such as errors 
occurring in measurements, errors occurring in recording and transmitting data, errors 
occurring by rounding etc. It is usually too complicated to model  noise deterministically, 
 so in practice, one often assumes that it  is random. In particular, a popular model is that 
the entries of $E$ are independent random variables with mean 0 and variance 1 (the value $1$  is, of course,  just matter of normalization).

For simplicity, we restrict ourselves to a representative case  when all entries 
of $E$ are iid Bernoulli random variables, taking values $\pm 1$ with probability half. For the treatment of more general models, see Section \ref{section:extension}.

\begin{remark} \label{remark:Bernoulli} We prefer the Bernoulli model over  the gaussian one
for two reasons.  
First, we believe that in  many real-life applications,
 noise must  have discrete nature (after all, data are finite). So it seems reasonable to use  
random variables with discrete support to model noise, and Bernoulli is the simplest such a variable.   Second,  as the reader will see, the analysis  for the Bernoulli model easily  extends to many other models of random matrices (including the gaussian one). On the other hand, 
the analysis for gaussian matrices often relies on  special properties of the gaussian measure  which are not available in other cases. \end{remark}

We say that an event $\CE$ holds almost surely if $\P(\CE)=1 -o(1)$; in other words, the probability that $\CE$ holds tends to one as $n$ tends to infinity. 
It is well-known that the norm of a random Bernoulli matrix 
is  of order $\sqrt n$, almost surely (see Lemma \ref{lemma:normE}).  Thus, Theorem \ref{theorem:SS}  implies  
the following variant of Corollary \ref{epsC1}.

%\begin{corollary}  \label{corollary:SS} There are positive constants $c, C$ such that the following holds. 
%Let $E$ is a random Bernoulli matrix. If  

%\begin{equation} \label{lower2} C \sqrt n \le \delta \end{equation}, then with probability $1-o(1)$

%\begin{equation} \label{upper2} \sin \angle(v_1, v'_1 )   \le  c \frac{ \sqrt n  }{\delta }. \end{equation} 
%\end{corollary} 

\begin{corollary} \label{epsC2} 
For any given $\eps >0$, there is $C=C(\eps)>0$ such that if $\delta \ge C \sqrt{n} $, then 
with probability $1-o(1)$ 

$$ \sin \angle(v_1, v'_1 ) \le \epsilon. $$
\end{corollary}

\subsection{Low dimensional data and improved bounds}

In a large variety of  problems,
%(in particular those related to compress sensing \cite{CT} or \cite{IL}), 
the   data is of small dimension, namely, $r:= \rank A \ll n$. 
The main point that we would like to make in this paper is that 
  in this setting,  the lower  bound on $\delta$   can be significantly improved.  
Let us first present the following (improved)  variant of Corollary \ref{epsC2}.

\begin{corollary}  \label{epsC3} 
For any positive constant $\epsilon$ there is a  positive constant $C= C(\epsilon)$ such that the following holds. 
Assume that $A$ has rank $r \le n^{.99}$ and  $\frac{n} {\sqrt{r \log n}} \le \sigma_1$
and  $\delta \ge C \sqrt {r \log n}$. 
Then with probability $1-o(1)$

\begin{equation}  \sin \angle(v_1, v'_1)  \le \epsilon . \end{equation} 

\end{corollary}

This result shows that (under the given circumstances)  we can approximate
 $v_1$ closely (by  $v_1'$)  provided   $\delta \ge C \sqrt {r \log n}$,  improving 
  the previous assumption $\delta  \ge C \sqrt n $.  Furthermore, the 
  appearance of $\sigma_1$ in the statement is necessary. If $\sigma_1 \ll \sqrt n$, then the noise 
dominates and we could not expect to recover any good information about $A$ from $A+E$.

%\begin{theorem}  \label{corollary:main1} 
%For any positive constant $\alpha_1$ there are positive constants $C, c$ such that the following holds. 
%Assume that $A$ has rank $r \le n^{1-\alpha_1}$ and  $\frac{n} {\sqrt{r \log n}} \le \sigma_1$.
%Let  $E$ is a random Bernoulli matrix. If 

%\begin{equation} \label{lower4} C \sqrt {r \log n} \le \delta \end{equation} 
 %then we have almost surely 

%\begin{equation} \label{upper4} \sin ^2 \angle(v_1, v'_1)  \le c
%\max \{  \frac{ \sqrt{r \log n} }{\delta },  (\frac{\sqrt n}{\sigma_1})^{3/2}  \} . \end{equation} 

%\end{theorem} 

Corollary  \ref{epsC3} is an easy  consequence of the following theorem.

\begin{theorem}  \label{theorem:main1} (Probabilistic $\sin$-theorem) 
For any positive constants $\alpha_1, \alpha_2$ there is a positive constant $C$ such that the following holds. 
Assume that $A$ has rank $r \le n^{1-\alpha_1}$ and  $\sigma_1:=\sigma_1(A)  \le  n^{\alpha_2} $. 
Let  $E$ be a random Bernoulli matrix.  Then with probabilty $1-o(1)$ 

\begin{equation} \label{upper3} \sin ^2 \angle(v_1, v'_1)  \le
 C \max \Big ( \frac{ \sqrt{r \log n} }{\delta },  \frac{n} {\delta \sigma_1}, \frac{\sqrt n}{\sigma_1}   \Big ) . \end{equation} 
 
 Furthermore, one can remove the term $\frac{\sqrt n}{\sigma_1} $ if $\delta \le \frac{1}{2} \sigma_1$.
\end{theorem}

Let us know consider the general case when we try to approximate the first $k$ singular vectors. 
Set $\eps_k:= \sin  \angle (v_k, v_k') $ and 
$s_k: = (\eps_1^2 + \dots + \eps_k^2)^{1/2}$. We can bound $\eps_k$ recursively as follows.

\begin{theorem} \label{theorem:main2} For any positive constants $\alpha_1, \alpha_2, k$ there is a positive constant $C$ such that the following holds. 
Assume that $A$ has rank $r \le n^{1-\alpha_1}$ and  $\sigma_1:=\sigma_1(A)  \le  n^{\alpha_2} $. 
Let  $E$ be a random Bernoulli matrix.  Then with probabilty $1-o(1)$

\begin{equation} \label{upper4} \eps_k^2  \le C  \max \Big( \frac{\sqrt {r \log n}} {\delta_k}, \frac{n}{\sigma_k \delta_k}, 
\frac{ \sqrt n}{\sigma_k} , \frac{\sigma_1^2 s_{k-1}^2}{\sigma_k \delta_k } , 
\frac{(\sigma_1+\sqrt n)(\sigma_k +\sqrt n)s_{k-1} }{\sigma_k \delta_k}  \Big). \end{equation} 

\end{theorem} 

The first three terms in the RHS of \eqref{upper4} mirror those in \eqref{upper3}. The last two terms represent the 
recursive effect.

To give the reader a feeling about this bound, let us consider the following example. Take $A$ such that $r  =n^{o(1)}$, $\sigma_1 =2 n^{\alpha}, \sigma_2 = 
n^{\alpha}, \delta_2 = n^{\beta}$,  where $\alpha > 1/2 > \beta > 1 -\alpha $ are positive constants.
Then $\delta_1 = n^{\alpha}$ and 
$\epsilon_1^2 \le  \max \Big( n^{-\alpha +o(1)} , n^{1 -2\alpha +o(1)}  \Big)$, almost surely. 

Assume that we  want to bound $\sin \angle (v_2, v_2')$. 
The gap $\delta_2  =n^{\beta} = o(n^{1/2})$, so Wedin theorem (in the general form) 
does not apply. On the other hand, 
Theorem \ref{theorem:main2} implies that almost surely 

$$\eps_2^2 \le \max \Big( n^{-\beta +o(1)}, n^{1/2 -\alpha +o(1)} , n^{-\alpha-\beta +1} \Big). $$

Thus, we have almost surely

$$\sin \angle (v_2, v_2')  = n^{-\Omega (1)} = o(1). $$

{\it The angle between two subspaces.} Let us mention that if $\sin \angle (v_j , v_j')  \le \eps$ for all 
$1 \le j \le k$, then $\sin \angle (V_k , V_k') \le \eps$, where $V_k$ ($V_k'$) is the subspace spanned by 
$v_1, \dots, v_k$ ($v_1', \dots, v_k '$, respectively).  The formal (and a bit technical) definition of $\angle (V_k, V_k')$ can be seen in \cite{GLbook, Bbook}. It is important to know that 
for two subspaces $V, V'$ of the same dimension

$$\sin \angle (V , V') = \| P - P' \| $$ where $P$ ($P'$) denotes the orthogonal projection onto 
$V$  ($V'$). Moreover $\|P-P'\|$  is frequently used as the distance between $V$ and $V'$.

\vskip2mm

The rest of the paper is organized as follows. In the next section, we present tools from linear algebra and probability.  The proofs of Theorems \ref{theorem:main1} 
and  \ref{theorem:main2} follow in Sections \ref{section:Proof1} and \ref{section:Proof2}, respectively. 
In Section \ref{section:extension}, we extend these theorems for other models of random noise, including 
the gaussian one, and also to matrices $A$ which do  not necessarily have low rank.

\section{Preliminaries: Linear Algebra and Probability}  \label{section:lemma}

\subsection{Linear Algebra} Fix a system  $v_1, \dots, v_n$ of unit singular vectors of $A$. 
It is well-known that  $v_1, \dots, v_n$ form an orthonormal basis.  (If $A$ has rank $r$, the choice of $v_{r+1}, \dots, v_n$ will turn out to be irrelevant.)

For a vector $v$, if we  decompose  it as 

$$v: =\alpha_1 v_1 + \dots +  \alpha _n v_n ,$$ then 

\begin{equation}
\label{eqn:basic}  \| Av\| ^2 = v \cdot A^{\ast} A v = \sum_{i=1}^n \alpha_i^2 \sigma_i^2  . \end{equation} 

We will use the 
Courant-Fisher minimax principle for singular values, which asserts that

\begin{equation} \label{minimax} \sigma_k  (M) = \max_{\dim H=k} \min _{ v\in H, \|v\| =1} \| Mv\| , \end{equation} where 
$\sigma_k(M)$ is the $k$th largest singular value of $M$.

%\begin{corollary} \label{corollary:1} 
%Let $v$ be a vector orthogonal to $v_1, \dots, v_{k-1}$. Write $v$ as 

%$$v = c_k v_k + \oc _k u $$ where $u$ is a unit vector  orthogonal to $ v_k$, then 

%$$\| Av \|^2 \le c_k^2 \sigma_k^2 + \oc _k ^2  \sigma_{k+1} ^2. $$

%\end{corollary} 

%We next consider the case when $v$ is near orthogonal to $v_1, \dots, v_{k-1}$. Assume that 
%$\|v\|=1$ and $| v \cdot v_i | \le \epsilon_i$ for all $1\le i \le k-1$. Write $v = c_k v_k +  \oc _k u$ where 
%$u$ is a unit vector orthogonal to $v_k$. We have 

%$$u =\alpha_1 v_1 + \dots + \alpha_n v_n . $$

%It is clear that $\alpha_k=0$. Furthermore, the near-orthogonality assumption implies that 
%for $1\le i \le k-1$

%$$|\oc _k\alpha_i| \le \epsilon. $$ 

%Now we have 

%$$\| Av \|^2 =  c_k^2 \sigma_k^2 + \oc _k ^2 \sum_{i \neq k} \alpha_i^2 \sigma_i^2  \le c_k^2 \sigma_k^2 +  
 %\sum_{i=1}^{k-1}  \epsilon_i^2 \sigma _i^2  + \oc _k ^2  \sigma_{k+1}^2.$$

%\begin{corollary}
%Let $v$ be a unit vector such that  $|v \cdot v_i| \le \epsilon_i$ for all $1\le i \le k-1$. Write $v$ as 

%$$v = c_k v_k + \oc _k u $$ where $u$ is a unit vector  orthogonal to $ v_k$, then 

%$$\| Av \|^2 \le c_k^2 \sigma_k^2 + \oc _k  ^2  \sigma_{k+1} ^2 +   \sum_{i=1}^{k-1}  \epsilon_i^2 \sigma_i ^2  . $$

%\end{corollary} 

%\end{document} 

\subsection{$\epsilon$-net lemma} 

Let $\epsilon$ be a positive number. A set $X$ is an $\epsilon$-net of a set $Y$ if for any $y \in Y$, there is $x \in X$ such that 
$\| x-y\| \le \epsilon$. 

\begin{lemma} \label{lemma:net} 
Let  $H$ be a subspace and  $S:= \{ v | \|v\| =1 , v \in H \}$. 
Let $0 <\eps  \le 1$ be a number and $M$ a linear map. Let $\CN \subset S$ 
be an ${\epsilon} $-net  $\CN$ of $S$. Then there is a vector $w \in \CN$ such that 

$$\| Mw\| \ge (1-\epsilon)  \max_{\|v\| \in  S } \| M v\| . $$

\end{lemma}

\begin{proof}
Let $v$ be the vector where the maximum is attained and let $w$ be a vector in the net closest to $v$ (tights are broken arbitrarily). 
Then by the triangle inequality 

$$\| Mw \|  \ge \| Mv\| - \| M (v-w) \| . $$

As $\| v-w\| \le \epsilon$, $\| M(v-w) \| \le \epsilon \max_{\|v\|  \in S} \| M v\| $, concluding the proof. 
\end{proof}

The following estimate for the minimum size of an $\epsilon$ of a sphere is well-known.
\begin{lemma} \label{lemma:sizeofnet} 
A unit sphere in $d$ dimension admits an $\epsilon$-net of size at most $(3\epsilon^{-1})^d$. 
\end{lemma} 

\begin{proof} Let $S$ be the sphere in question, centered at $O$, and $\CN \subset S$ be a finite subset of $S$ such that the distance between any two points is at least $\epsilon$. If  $\CN$ is maximal with respect to this property then $\CN$ is an $\epsilon$-net. On the other hand, the balls of radius $\epsilon/2$ centered at the points in $\CN$ are disjoint subsets of 
the the ball of radius $(1+\eps/2)$, centered at $O$. Since 

$$\frac{1+\eps/2}{\eps/2}  \le 3 \eps^{-1} $$ the claim follows by a volume argument.  \end{proof}

\subsection{Probability}

We need the following estimate on $\|E \|$  (see \cite{AKV, Meckes}). 

\begin{lemma} \label{lemma:normE} There is a constant $C_0 >0$ such that the following holds. Let $E$ be a random Bernoulli matrix of size $n$. Then

$$\P(\|E \| \le 3 \sqrt n) \le \exp(- C_0 n). $$
\end{lemma} 

Next, we present a lemma which roughly asserts  that for any two vectors 
given $u$ and $v$, $u$ and $Ev$ are, with high probability, almost orthogonal. 
We present the proof of this lemma in \ref{section:Proofproduct}.

\begin{lemma}  \label{lemma:product}  
%There are constants $C_1, C_2$ such that the following holds.
 Let $E$ be a random Bernoulli matrix of size $n$. 
For any fixed unit vectors $u,v$ and positive number $t$

$$\P (| u^T E  v | \ge t )  \le 2 \exp (-t^2/16) .$$
\end{lemma}

%\begin{lemma} For any fixed unit vector $v$

%$$\| (A+E) v \| ^2 \le \| Av\| ^2 + 9n + O(\|A \| ). $$ \end{lemma} 

Now we are ready to state our key lemma. 

\begin{lemma} \label{lemma:subspace}
For any constants $0 < \beta_1, 0<  \beta_2 < 1$ there is a constant $C$ such that the following holds.  Assume that
$A$ is such that  $\sigma_1 \le n^{\beta_1}$ and let  $V :=\Span \{v_1, \dots, v_d\}$
 for some $d \ge n^{1-\beta_2}$. 
Then the following holds almost surely. For any unit vector $v \in V$

$$ \| (A+E) v\| ^2 \le \sum_{i=1}^n (v \cdot v_i)^2 \sigma_i^2 + C(n +\sigma_1 \sqrt {d \log n} ) . $$
\end{lemma} 

\begin{proof}
It suffices to prove for $v$ belonging to an $\eps$-net $\CN$ of the unit sphere $S$ in $V$, with $\eps: =\frac{1}{n+ \sigma_1}$. 
With such small $\eps$, the error coming from the term $(1-\eps)$ (in Lemma \ref{lemma:net}) 
is swallowed into the error term $O(n +\sigma_1 \sqrt {d \log n} )$. 

By Lemma \ref{lemma:net}, $|\CN| \le (\frac{3}{\eps})^d \le  \exp(C_1 d \log n) $, 
for some constant $C_1$ (which depends on the exponent $\beta_1$ in the upper bound of $\sigma_1$).
 Thus, using the union bound,  it suffices to show that if $C$ is large enough, then for any $v \in \CN$

$$\P (  \| (A+E) v\| ^2 \ge \sum_{i=1}^n (v \cdot v_i)^2 + C(n +\sigma_1 \sqrt{d \log n} ) ) \le \exp(-2C_1 d \log n ) $$ for any fixed $v \in \CN$. 

Fix $v \in \CN$. By \eqref{eqn:basic},  

$$ \| (A+E) v\| ^2=  \|Av\|^2   + \| Ev\| ^2 + 2 (Av) \cdot (Ev)  = \sum_{i=1}^n (v \cdot v_i)^2 \sigma_i^2 + \|Ev\|^2 + 2 (Av) \cdot (Ev). $$

Since $\|Av\| \le \sigma_1$, we have, by Lemma \ref{lemma:product}, that with probability at least  $1-\exp(-C_2 d \log n )$
$$|(Av) \cdot (Ev) | \le C \sigma_1 \sqrt {d \log n} ,$$ where $C_2$ increases with $C$. Thus, by 
choosing  $C$ sufficiently large, we can assume that 
$C_2 > 3C_1$. 

Furthermore, by Lemma \ref{lemma:normE},   $\| Ev\| \le  3\sqrt n$
with probability  at least $1-\exp(-\Omega (n))$. 
Combining this with the above bounds, we conclude that for a sufficiently large constant $C$

\begin{eqnarray*} 
 \P (  \| (A+E) v\| ^2 \ge \sum_{i=1}^n (v \cdot v_i)^2 + C(n +\sigma_1) ) 
&\le&  \exp(-3C_1 d \log n) +  \exp(-\Omega (n)) \\ &\le&  \exp(-2C_1 d \log n ) , \end{eqnarray*}

\noindent completing the proof. 
\end{proof}

\section {Proof of Theorem \ref{theorem:main1}}  \label{section:Proof1}

%\section{Preliminaries}

Let $H$ be the subspace spanned by $\{v_1, v_2\}$ and  $u_i (1 \le i \le n)$ be the singular vectors of 
the matrix $A^{\ast}$.

% Let $\CE_1$ be the event that 
%the conclusion of Lemma \ref{lemma:subspace} holds for $H$ and $\CE_2$ be the event that the norm bound in 
%Lemma \ref{lemma:norm} holds. From these lemmas, we know that 
%both $\CE_1$ and $\CE_2$ holds with very high probability

%$$\P (\CE_1\wedge \CE_2)  \ge 1 -\exp(-\Omega (n)) = 1-o(1). $$

%To prove the theorem, we  show that if  both $\CE_1$ and $\CE_2$ hold, then 
%$\sin \angle (v_1, v_1') $ is small. 

First, we  give a lower bound for  $\sigma_1' : =\| A+E\| $.  By the minimax principle, we have 

$$\sigma_1' =  \| A+E \| \ge | u_1^T (A+E) v_1 | =| \sigma_1  + u_1^T  E v_1| . $$

By Lemma \ref{lemma:product}, we have, with probability $1- o(1) $, 
$|u_1 ^T E v_1|  \le \log \log n $. (The choice of $\log \log n$ is not important. One can  replace it by any function that tends slowly  to infinity with $n$.)

Thus, 
we have, with probability $1- o(1) $, that 

\begin{equation} \label{eqn:lowerbound}  \| A+E \| \ge  \sigma_1  - \log \log n .  \end{equation}

Our main observation is that, with high probability,  any $v$ that is far from $v_1$ 
would yield $\| (A+E)v \| < \sigma_1 -\log \log n$. Therefore, the first singular vector $v_1'$ of 
$A+E$ must be close to $v_1$. 

%$$c_2^2 \ge  , $$ 

Consider a unit   vector $v$ and write it as 

$$v =c_1 v_1 +  c_2 v_2 +\dots + c_r v_r+  c_0u    $$ where $u$ is a unit vector  orthogonal to $H:= \Span \{v_1, \dots, v_r\}  $ and $c_1^2+ \cdots + c_r^2+ c_0^2 =1$. 
Recall that $r$ is the rank of $A$, so $Au =0$.
Setting $w:= c_1v_1+\dots + c_r v_r$ and using Cauchy-Schwartz, we have

\begin{eqnarray*} 
\| (A+E)v\|^2 &=& \| (A+E)w + c_0 E u \|^2  \le \| (A+E) w \|^2 + 2 c_0 \| (A+E) w\| \| Eu\| + c_0^2 \| Eu\| ^2  \\
&\le& (1+ \frac{c_0^2}{4}) \| (A+E) w\|^2 + (4 + c_0^2) \| Eu \|^2. 
\end{eqnarray*} 

By Lemma \ref{lemma:normE}, we have, with probability $1- o(1)$, that 
$\| Eu\| \le  3\sqrt n)$ for every  unit vector $u$. 
Furthermore, by Lemma  \ref{lemma:subspace}, we have, with probability $1 -o(1) $, 

$$ \|(A+E) w \|^2 \le  \sum_{i=1}^r (w \cdot v_i)^2  + O(\sigma_1 \sqrt {r \log n} +n ) $$ for every  vector $w \in H$ of length at most $1$.

Since 

$$ \sum_{i=1}^r (w \cdot v_i)^2 =\sum_{i=1}^r c_i^2 \sigma_i^2 \le (1-c_0^2) \sigma_1^2  -(1-c_0^2- c_1^2) (\sigma_1^2-\sigma_2^2) , $$ we can conclude that 
 with probability $1 -o(1) $,  the first singular vector of $A+E$, written in the form 
 $v=c_1v_1+ \dots + c_r v_r + c_0u$, satisfies 
 
 \begin{equation} \label{eqn:upperbound} 
\frac{1}{1+ c_0^2/4}   \| (A+E) v\|^2 \le (1-c_0^2) \sigma_1^2  -(1-c_0^2- c_1^2) (\sigma_1^2-\sigma_2^2) + O(\sigma_1  \sqrt{r \log n} + n ). \end{equation} 

Notice that $c_0\le 1$, so the term $c_0 n$ is  swallowed into  $O(n)$. 
By \eqref{eqn:lowerbound} and the fact that 
$\frac{1}{1+ c_0^2} \ge 1 -\frac{c_0^2}{4}$, we have 

$$\frac{1}{1+ c_0^2/4}   \| (A+E) v\|^2 \ge (1 -\frac{c_0^2}{4}) (\sigma_1 -\log \log n)^2. $$

Comparing this with \eqref{eqn:upperbound} 
and noticing that both $\sigma_1 \log \log n $ and $(\log \log n)^2$ are $o( \sigma_1 \sqrt {r \log n})$, 
we obtain, for some properly chosen  constant $C$, that 

$$(1-c_1^2) \sigma_1 \delta   -\frac{c_0^2}{4} \sigma_1^2  \le -c_0^2 \sigma_2^2 +C (\sigma_1 \sqrt {r \log n}  + n ). $$

Before concluding the proof, let us derive a bound on $c_0$. We can show that with probability $1-o(1)$

\begin{equation} \label{boundc0} c_0^2 =O(\frac{\sqrt n}{\sigma_1}). \end{equation} 

To verify this, we again used the bound  $\| (A+E) v\| \ge \sigma_1 -\log \log n$. 
Oh the other hand, by the triangle inequality and Lemma \ref{lemma:normE}, 
we have with probability $1-o(1)$

$$\| (A+E) v\| \le \| Av\| + \|Ev\| \le \sqrt {1- c^2} \sigma_1 + 3 \sqrt n,$$ 
from which \eqref{boundc0} follows by a simple computation.

Without loss of generality, we can assume that $C\ge 1$. If 
$\sigma_2 \le \frac{1}{2} \sigma_1$, then $\delta \ge \frac{1}{2}\sigma_1$ and 

\begin{equation}  
\label{bound1} 1 -c_1^2 \le \frac{ C (\sigma_1 \sqrt {r \log n}  + n)}{\sigma_1^2/2}  +\frac{c_0^2}{2} 
 =O(\frac{\sqrt r \log n}{\sigma_1})  + O(\frac{n}{\sigma_1^2} ) + O(\frac{\sqrt n}{ \sigma_1}). 
 \end{equation} 

In the case   $\sigma_2  \ge \frac{1}{2} \sigma_1$, $c_0^2 \sigma_2^2 \ge \frac{c_0^2}{4} \sigma_1^2 $, so

$$(1-c_1^2) \sigma_1 \delta \le C (\sigma_1 \sqrt {r \log n} +n) $$ which implies 

\begin{equation}  \label{bound2} 
(1-c_1^2) \le C (\frac{\sqrt{r \log n} }{\delta} + \frac{ n}{\sigma_1 \delta}) . \end{equation}

Notice that  $\sin \angle (v_1,v_1') ^2 =  \sin\angle (v_1 , v)^2 =  1- c_1^2 $. 
The desired claim follows from \eqref{bound1} and \eqref{bound2}.

\begin{remark} One can improve the error term $\frac{\sqrt n}{\sigma_1}$ to $(\frac{\sqrt n}{\sigma_1})^{3/2}$. However, this proof is more technical and harder to generalize. \end{remark} 

\section{Proof of Theorem \ref{theorem:main2}}   \label{section:Proof2}

Similar to the previous proof, we start with  a lower bound for 
$\sigma_k'$, the $k$th largest singular value of $A+E$. Using the minimax principle, we have 

$$\sigma_k' \ge |u_k^{T} (A+E) v_k | \ge \sigma_k - \log \log n $$ with probability $1-o(1)$.

We need to consider $\| (A+E) v \|$ for a unit vector $v$ orthogonal to 
$v_1' , \dots, v_{k-1}'$. We write (as before) 

$$v := c_{1}  v_1 + \dots + c_{ r} v_r +  c_0 u = w + c_0 u . $$

If $v$ is the $k$th singular vector of $A+E$, then  $v \cdot v_j'= 0$ for $1 \le j \le k-1$,  and we obtain
$$|c_j| =  |v \cdot v_j| = | v \cdot (v_j- v_j') | \le |v_j -v_j'| \le  2 \sin \angle (v_j. v_j')  = 2  \eps_j . $$

As in the previous proof, we consider the inequality

\begin{eqnarray*} 
\| (A+E)v\|^2 &=& \| (A+E)w + c_0 E u \|^2  \le \| (A+E) w \|^2 + 2 c_0 \| (A+E) w\| \| Eu\| + c_0^2 \| Eu\| ^2  \\
&\le& (1+ \frac{c_0^2}{4}) \| (A+E) w\|^2 + (4 + c_0^2) \| Eu \|^2. 
\end{eqnarray*}

We split $w= \bar w_k + w_k$, where 
 $\bar w_k := c_1 v_1 + \dots + c_{k-1} v_{k-1}$ and $ w_k: =c_{k} v_k + \dots v_r c_r$.
We have 

$$\| (A+E) w\|^2 = \|(A+E)(\bar w_k +w_k)\|^2 \le 
\| (A+E) w_k \|^2 + \| (A+E) \bar w_k\|^2 + 2 \| (A+E) w_k\| \| (A+E) \bar w_k \|. $$

Using Lemma \ref{lemma:subspace}, we have 

\begin{equation} \| (A+E) w_k \|^2 \le c_k^2 v_k^2 + \dots + c_r^2 v_r^2 + O( \sigma_k \sqrt{r \log n} + n).
\end{equation} 

The term $\| (A+E) \bar w_k \|^2$ can be bounded, rather generously, by 

\begin{equation}  O( (\sigma_1 + \sqrt n)^2  (c_1^2 + \dots + c_{k-1}^2 )
= O(\sigma_1^2 + n) s_{k-1}^2) . \end{equation}  

Moreover, 

\begin{equation} \| (A+E) w_k\| \| (A+E) \bar w_k \| = 
O( (\sigma_k +\sqrt n)(\sigma_1+ \sqrt n) \| w_k\| \| bar w_k \| =O((\sigma_1+\sqrt n)(\sigma_k +\sqrt n)s_{k-1}  .
\end{equation}

Repeating the calculations in the previous proof, we have, with probability $1-o(1)$

$$(1-c_k^2) (\sigma_k^2 -\sigma_{k+1}^2)  -\frac{c_0^2}{4} \sigma_k^2 
\le \sum_{j=1}^{k-1} c_j^2 (\sigma_j^2 -\sigma_{k+1}^2) 
-c_0^2 \sigma_{k+1}^2 + O( \sigma_k \sqrt {r \log n}+ n )  + O(\sigma_1^2 s_{k-1}^2 +
(\sigma_1+\sqrt n)(\sigma_k +\sqrt n)s_{k-1} ) . $$

We can bound $c_0$ as follows 

\begin{equation} 
c_0^2 + s_{k-1}^2  =O( \frac{\sqrt n}{\sigma_k} + \frac{\sigma_1 s_{k-1} } {\sigma_k} ). 
\end{equation}

By considering the two cases $\sigma_{k+1} \ge \frac{1}{2}\sigma_k$ and 
$\sigma_{k+1} < \frac{1}{2} \sigma_k$,  the desired bound follows.

\section{Extensions} \label{section:extension}

In this section, we extend our results to other models of random matrices. It is easy to see that we 
did not rely ver heavily on properties of the Bernoulli random variable. All we need is a model of random matrices 
so that Lemmas \ref{lemma:normE} and \ref{lemma:product} (or sufficiently strong variants) hold. 

Both of these lemmas hold for the case where the noise is gaussian (instead of Bernoulli).  In fact, Lemma \ref{lemma:product} is trivial as $u^TEv$ has distribution  $N(0,1)$.

Both lemmas hold in the case the entries of $E$ is bounded by a universal constant $K$. For the proof of
Lemma \ref{lemma:normE}, see \cite{AKV, Meckes}. For the proof of Lemma \ref{lemma:product}, see 
Remark \ref{remark:product}. 

Quite often, the boundedness condition can be replaced by the condition of having a 
rapidly decaying tail (such as sub-gaussian), using either more advanced concentration 
tools (see \cite{Talbook}) or a truncation argument (see \cite{TVhard}). We do not pursuit these matters here.

We can also extend our results for a matrix $A$ which does not have low rank, but can be well approximate by one. In this case, we consider  $A= A'+B$, where $A'$ has small rank (say $r$) and $B$ is very small. In this case, we  can apply, say,  Theorem  \ref{theorem:main1} to bound $\|v_1 (A')-v_1(A'+E)\|$ and Theorem \ref{theorem:SS}  to bound $\| v_1(A') - v_1(A)\|$ and then use the triangle inequality.  
As a result, the RHS of \eqref{upper3} will have an extra term $\frac{\| B\| }{\delta}$. The reader is invited to work out the details. 

Finally, our analysis also extends fairly easily to the case when $E$ is a hermitian random matrix
(either Wigner or Wishart model)  and $A$ is hermitian. The details and few applications will appear elsewhere.

\appendix 

\section{Proof of Lemma \ref{lemma:product}} \label{section:product}

As $u^T E v = \sum_{i,j} u_j v_j \xi_{ij}$ where $u=(u_i)_{i=1}^n, v=(v_j)_{j=1}^n$ and the $\xi_{ij}$ 
are the entries of $E$, Lemma \ref{lemma:product} follows from 

\begin{lemma}  \label{lemma:sum} 
Let $S:= c_1 \xi_1 + \dots + c_n \xi_n$ where $\xi_i$ are iid Bernoulli random variables and $c_i$ are real numbers such as $\sum_{i=1}^n c_i^2 =1$. Then for any number $t > 0$

$$\P( |S|  \ge t ) \le 2 \exp (- t^2 /16). $$

\end{lemma}

\begin{proof} 
Without loss of generality, we can assume that $|c_i|$ decreases and $l$ is the last index such that 
$|c_i| \ge \frac{2}{T}$.  As  $\sum_{i=1}^n c_i^2 =1$, $l \le t^2/4$. 
By Cauchy-Schwartz, 

$$| c_1 \xi_1 + \dots + c_l \xi_l|^2 \le l^2 \sum_{i=1}^n c_i^2 \le \frac{t^2}{4} , $$ 
which implies that with probability one $|c_1\xi_1 + \dots c_l \xi_l| \le \frac{t}{2} $. Therefore,

$$\P( |S| \ge t) \le \P( |S'| \le \frac{t}{2} ), $$ where $S':= \sum_{i=l+1}^n c_i \xi_i$. 

We can bound $ \P( |S'| \le \frac{t}{2} )$ by the standard Laplace-transform argument. Set $z:= t/4$. 
Thanks to independence, we  have 

$$\P( S' \ge \frac{t}{2}) = \P( \exp( zS') \ge e^{tz/2} ) \le e^{-tz/2} \E (\exp(zS') )= 
e^{-tz/2} \prod_{i=l+1}^n \E \exp(zc_i \xi_i)
. $$

On the other hand, as $|zc_i| \le 1$, it is easy to show that

$$\E \exp( zc_i \xi_i)  \le 1 + (zc_i)^2 \le \exp( z^2 c_i^2 ). $$

Together, we obtain 

$$\P ( S' \ge tz/2) \le e^{-tz/2} \exp( \sum_{i=l+1}^n z^2 c_i^2) \le \exp( z^2 -\frac{tz}{2})  =\exp( -\frac{t^2}{16}). $$

Similarly 

$$\P( S' \le -tz/2) = \P( -S' \ge tz/2) \le \exp(-\frac{t^2}{16}) , $$ concluding the proof. 
\end{proof}

\begin{remark}  \label{remark:product} 
The same proof works for $\xi$ being  arbitrary independent random variable with mean 0 and variance 1, 
uniformly bounded by a constant $K$. In this case, the constant $16$ is replaced by a constant depending on $K$. 
\end{remark}

\end{document}